\documentclass[journal]{new-aiaa}
\usepackage[utf8]{inputenc}
\usepackage{textcomp}

\usepackage{graphicx}

\usepackage{amsmath}
\usepackage[version=4]{mhchem}
\usepackage{siunitx}
\usepackage{longtable,tabularx}
\usepackage{alphabeta}
\usepackage{xcolor}
\usepackage{booktabs}

\usepackage{amsthm}

\newtheorem{lemma}{Lemma}
\newtheorem{corollary}{Corollary}
\newtheorem{proposition}{Proposition}
\theoremstyle{definition}
\newtheorem{definition}{Definition}
\theoremstyle{remark}
\newtheorem{remark}{Remark}

\usepackage{tikz,pgfplots}
\usetikzlibrary{calc,intersections,through,backgrounds, math, arrows.meta}
\usetikzlibrary{external}

\usepackage{tikzscale}
\tikzset{no slope/.code={\pgfslopedattimefalse}}
\tikzset{vect/.style={->,shorten >=0.5pt,>=latex, line width = 0.2}}
\tikzset{point/.style={circle, fill, inner sep=1.7}}
\pgfplotsset{compat=1.18}
\usepgfplotslibrary{patchplots}
\usepgfplotslibrary{fillbetween}
\pgfplotsset{%
    layers/standard/.define layer set={%
        background,axis background,axis grid,axis ticks,axis lines,axis tick labels,pre main,main,axis descriptions,axis foreground%
    }{
        grid style={/pgfplots/on layer=axis grid},%
        tick style={/pgfplots/on layer=axis ticks},%
        axis line style={/pgfplots/on layer=axis lines},%
        label style={/pgfplots/on layer=axis descriptions},%
        legend style={/pgfplots/on layer=axis descriptions},%
        title style={/pgfplots/on layer=axis descriptions},%
        colorbar style={/pgfplots/on layer=axis descriptions},%
        ticklabel style={/pgfplots/on layer=axis tick labels},%
        axis background@ style={/pgfplots/on layer=axis background},%
        3d box foreground style={/pgfplots/on layer=axis foreground},%
    },
}

\hypersetup{urlcolor=purple, citecolor=red, linkcolor=blue, pdftitle={Basic Engagement Zones}, pdfauthor={Von Moll, Weintraub}}

\usepackage[nameinlink]{cleveref}

\crefmultiformat{equation}{Eqs.~(#2#1#3)}%
{ and~(#2#1#3)}{, (#2#1#3)}{ and~(#2#1#3)}
\crefrangeformat{equation}{Eqs.~(#3#1#4)--(#5#2#6)}
\crefformat{equation}{Eq.~(#2#1#3)}

\crefmultiformat{figure}{Fig.~#2#1#3}%
{ and~Fig.~#2#1#3}{, Fig.~#2#1#3}{ and~Fig.~#2#1#3}
\crefrangeformat{figure}{Figs.~#3#1#4 to~#5#2#6}
\crefformat{figure}{#2Fig.~#1#3}
\Crefformat{figure}{#2Fig.~#1#3}

\crefrangeformat{lemma}{Lemmas~#3#1#4--#5#2#6}
\crefrangeformat{theorem}{Theorems~#3#1#4--#5#2#6}

\DeclareMathOperator*{\atan}{atan2}

\usepackage{subfigure}

\newcommand\blfootnote[1]{%
  \begingroup
  \renewcommand\thefootnote{}\footnote{#1}%
  \addtocounter{footnote}{-1}%
  \endgroup
}

\title{Basic Engagement Zones}

\author{Alexander Von Moll\footnote{Aerospace Engineer} and Isaac Weintraub\footnote{Electronics Engineer, AIAA Senior Member}}
\affil{Air Force Research Laboratory, 2210 8th St. WPAFB, OH 45433 USA}

\begin{document}

\maketitle

\section{Introduction}
\label{sec:introduction}

\lettrine{N}avigation in threat-laden environments is a fundamental problem for many applications and missions~\cite[pp.~III-1]{scott2017countering}.
\blfootnote{
    This paper is based on work performed at the Air Force Research Laboratory (AFRL) \textit{Control Science Center}.
    Distribution Statement A: Approved for Public Release;
    Distribution is Unlimited.
    PA\# AFRL-2023-5574.}
For example for a supply airplane to reach its destination in a reasonable time, it may have to fly over or near unfriendly areas containing one or more ground-based enemy defensive assets, which could fire upon the airplane should it come too close.
Ideally, the airplane and its crew strives to reach a designated goal as quickly as possible, while avoiding aggressive maneuvers to avoid interception, should one of the defensive assets decide to engage.
While there is a wealth of literature concerning navigation in obstacle-laden environments (c.f., e.g., ~\cite{patle2019review:}), there is a gap concerning the case when the ``obstacles'' are not only dynamic but also \textit{adversarial}.

When it comes to modeling adversarial engagements between mobile agents, a common paradigm has been differential games (c.f., e.g., \cite{isaacs1965differential}).
If a scenario can be described wherein two players have a diametrically opposed objective that can be distilled to a single scalar value, one can formulate a zero-sum differential game.
The solution of a differential game may be comprised of the Value function that describes the equilibrium value of the cost/reward as a function of the initial condition as well as the equilibrium control policies for the two players (often, the Nash equilibrium is the desired solution concept).
Two particular classes of differential game problems are particularly relevant to the adversarial navigation problem described above: \textit{pursuit-evasion} and \textit{target guarding} (which also appears in the literature as \textit{reach-avoid}).
In the former, a Pursuer generally seeks to capture an Evader in minimum time or minimum control effort, while in the latter, a Defender generally seeks to capture an Attacker before the latter is able to reach some target or goal region; many of these problems and the literature surrounding them are described in~\cite{weintraub2020introduction} .
However most of these solutions, in themselves, are not well-suited to address the navigation problem.
With pursuit-evasion solutions, for example, it's difficult to incorporate objectives for the Evader outside of merely evading without completely changing the problem formulation.
Target Guarding solutions, on the other hand, typically delineate where breach or capture can be guaranteed but generally do not account for possible constraints on the Defender.

Some notable and relevant exceptions include the following works.
In~\cite{chen2022geometric} the authors explicitly formulated a target guarding game with a fixed time limit.
This could be seen as a constraint on the Defender's onboard fuel, for example.
However, the solution only answers the question of what is the outcome (breach or capture) and equilibrium headings of the agents as a function of the initial positions assuming that the engagement (i.e., the playout of the differential game) begins immediately.
In other words, it does not directly address the strategy of either agent before or after the engagement.
Another example is the work in~\cite{patil2023risk-minimizing}.
There, a zero-sum stochastic differential game is formulated in which one player desires to keep the state of the system in the safe set for some specified duration (akin to a range or fuel constraint) while the other desires to cause the system to enter the unsafe set.
Like the previous work, this solution describes what the agents should do once the clock starts but does not directly address how the Evader should navigate to avoid (\textit{a priori}) getting into a situation in which it could lose this game.

In order to begin addressing what may happen before or after an engagement takes place we introduce the following.
\begin{definition}[Engagement Zone]
    \label{def:engagement_zone}
    Given a Mobile Agent, $A$, Threat, $T$, their respective dynamic models, $\dot{\mathbf{x}}_A$, $\dot{\mathbf{x}}_T$, and an Agent strategy, $u_A(t)$, an Engagement Zone (EZ) is a region of the state space in which it is possible for the Threat to neutralize the Mobile Agent if the latter does not deviate from its current strategy.
\end{definition}
Note this is just one particular definition for EZ.
The interpretation is this: if the Agent remains outside of the EZ associated with a particular Threat that the latter may not have any incentive to actively engage the Agent as 1) neutralization (e.g., capture) cannot be guaranteed and 2) the Agent does not even need to actively maneuver/evade to avoid the Threat.
The latter point distinguishes the EZ from the less conservative approach of win/lose regions obtained from solutions of differential games.
Navigating near the win/lose region boundary could necessitate an aggressive maneuver by the Agent should the Threat decide to engage.

Some earlier papers by the authors have begun to address navigation around EZs.
In~\cite{weintraub2022optimal} the single-Agent, single-EZ navigation problem was introduced.
There, path plans were devised for which the Agent avoided entering into the EZ and others which allowed some penetration into the EZ to reduce overall travel time or reach the goal at a specified arrival time.
Then the results were extended in~\cite{dillon2023optimal} to address navigation of a single vehicle around (or through) two EZs.
In both of these works, a notional EZ model based on a cardioid shape was assumed.
This EZ was not tied to a particular model of the Threat itself.
In practice, for a particular Agent and Threat model, one may compute EZs via simulation, in which case the EZ may be data-based and not analytic which could make planning a path more computationally expensive.
This paper makes the following two contributions:
1) a set of EZs that are based on first-principle models of the Agent and Threat
2) a direct comparison of EZ-based navigation and path-planning to the nominal approach of circumnavigation.

The remainder of the paper is organized as follows.
\Cref{sec:Pursuit-Evasion,sec:Turret-Evasion} model the Threat as a Pursuer and a Turret, respectively, and derive the associated EZs.
\Cref{sec:Application_to_Path_Planning} highlights path planning as a potential application of EZs.
\Cref{sec:conclusion} concludes the paper.

\section{Pursuit-Evasion}
\label{sec:Pursuit-Evasion}

\subsection{Basic Pursuit-Evasion Model}
\label{sec:Basic_Pursuit-Evasion_Model}

In this section, the Threat from~\Cref{def:engagement_zone} is specialized to be a Pursuer.
Perhaps the most basic dynamic model to consider for pursuit-evasion is that of both agents moving with simple motion in the two-dimensional plane (i.e., $A, P \in \mathbb{R}^2$).
That is, both agents move with constant speeds and have control over their instantaneous heading~\cite{weintraub2023range-limited}:
\begin{equation}
    \label{eq:f_pursuit_evasion}
    \dot{\mathbf{x}} = \begin{bmatrix} \dot{\mathbf{x}}_P \\ \dot{\mathbf{x}}_A\end{bmatrix} =
    \begin{bmatrix}
        \dot{x}_P \\
        \dot{y}_P \\
        \dot{x}_A \\
        \dot{y}_A 
    \end{bmatrix} =
    \begin{bmatrix}
        v_P \cos ψ_P \\
        v_P \sin ψ_P \\
        v_A \cos ψ_A \\
        v_A \sin ψ_A
    \end{bmatrix},
    \qquad
    \mathbf{x}(0) =
    \begin{bmatrix}
        x_{P_0} \\
        y_{P_0} \\
        x_{A_0} \\
        y_{A_0}
    \end{bmatrix}
\end{equation}
where the Mobile Agent's control is $u_A \equiv ψ_A$.
For this scenario, neutralization is said to occur if $A$ comes within $r\geq 0$ distance of $P$ (i.e., $A$ is captured by $P$):
\begin{equation}
    \label{eq:neutralization_pursuit_evasion}
    \mathcal{N} = \left\{ \mathbf{x} \mid \overline{AP} \leq r \right\},
\end{equation}
where $\mathcal{N}$ is the neutralization set, and the notation $\overline{AP}$ is used to indicate the distance between two points, e.g., $A$ and $P$.
Define the speed ratio of Mobile Agent and Pursuer speeds as $μ = \tfrac{v_A}{v_P} > 0$.
Note that since the EZ is based on the Mobile Agent implementing zero control capture may be possible, depending on the initial conditions, even when $μ > 1$.
It is common to consider such a pursuit-evasion scenario as-is without any additional constraints.
Alternatively, some work has been done on similar scenarios taking place in the presence of obstacles~\cite{oyler2016pursuit-evasion}, in bounded domains~\cite{zhou2016cooperative}, and under a variety of other geometric and integral constraints~\cite{ibragimov1998game}.
This paper considers the Pursuer to be range-limited with maximum range $R$ as in~\cite{weintraub2023range-limited}.
Thus the Pursuer is capable of reaching any point within the disk of radius $R$ centered at its initial position, and thus its reachability region is
\begin{equation}
    \label{eq:P_reachability}
    \mathcal{R}_P = \left\{ (x,y) \mid \left(x-x_{P_0}\right)^2 + \left( y - y_{P_0} \right)^2 \leq R^2\right\}.
\end{equation}
Note that one possible interpretation of the range-limit is that the Pursuer may have arbitrary endurance as long as it stays within $R$ of its initial position.
Under this interpretation, the range-limit may represent a maximum communication range or sensing range for some base station located at the Pursuer's initial position.
Another interpretation is that the Pursuer may only traverse a total distance $R$ along its trajectory, which does not necessarily preclude the Pursuer from stopping.
Based on the constant speed model of the Pursuer, however, stopping would require, e.g., modulating its heading, $ψ_P$, infinitely fast.
In this work, such a maneuver, although mathematically possible, is considered to consume the Pursuer's remaining range at the same rate as if it were moving.
Therefore, the range-limit is analogous to a time-limit -- the Pursuer may be equipped with, for example, a solid rocket motor which, once started, has a fixed time at which fuel will run out (barring changes in environmental conditions or changes in various forces experienced along the trajectory).

\begin{remark}
    If the Pursuer has additional, unmodeled dynamics (e.g., a bounded turn rate) then this analysis is conservative.
\end{remark}

\subsection{Pursuit-Evasion EZ}
\label{sec:Pursuit_Evasion_EZ}

\begin{lemma}
    Under the model~\Cref{eq:f_pursuit_evasion} if the Mobile Agent, starting from $A_0$ and moving with $ψ_A$ cannot be neutralized by $P$ under collision course guidance then the point $A_0$ is outside of the EZ.
\end{lemma}
\begin{proof}
    Collision course, wherein the Pursuer takes a straight-line path along the line connecting its initial position to the Mobile Agent's position at the time of neutralization, yields the minimum time trajectory for $P$ to intercept $A$~\cite{weintraub2023range-limited}.
    Therefore, if the collision course trajectory results in $P$ travelling a distance greater than $R$, then the Mobile Agent cannot be neutralized under any guidance law.
    Thus, from~\Cref{def:engagement_zone}, the initial position of $A$ must be outside the EZ.
\end{proof}

\begin{corollary}
    The boundary of the EZ is the locus of all Mobile Agent initial positions corresponding to neutralization by the Pursuer via a collision course trajectory of length equal to its maximum range, $R$.
\end{corollary}

This result is useful in obtaining the boundary of the EZ as follows.
Define the quantity $ρ \equiv \overline{P_0 A_0}$ to be the initial distance between the agents resulting in the Pursuer taking a collision course trajectory of length $R$.
Then the Engagement Zone is defined mathematically by
\begin{equation}
    \label{eq:Z_pursuit}
    \mathcal{Z} = \left\{ A_0 \mid \overline{P_0 A_0} \leq ρ(ξ; μ, R, r) \right\},
\end{equation}
where $ξ$, the aspect angle, is the angle the Agent's heading makes with the initial line-of-sight angle.
Define a set of coordinate axes $(\hat{x}, \hat{y})$ whose origin is $P_0$ in which the $\hat{x}$ vector is aligned with the heading of the Mobile Agent.

\begin{remark}
    This model and the resulting EZ may be the simplest way of directly accounting for the effect of aspect angle and closing speed on the engagement which was the main physical phenomena under investigation in~\cite{weintraub2022optimal}.
\end{remark}

\subsubsection{Fast Pursuer}

For the case of a fast Pursuer ($μ \leq 1$), the Law of Cosines applied to the triangle $\triangle P_0 E_0 E_f$ may be used to obtain an expression for $ρ$ (see~\Cref{fig:EZ_fast_Pursuer}):
\begin{equation}
    \label{eq:ρ_pursuit_fast}
    ρ(ξ; μ, R, r) = μ R \left[ \cos ξ + \sqrt{\cos^2 ξ - 1 + \frac{(R + r)^2}{μ^2 R^2}} \right], \qquad μ \leq 1,\ ξ \in \left\{ -π, π \right\}.
\end{equation}
\Cref{fig:EZ_fast_Pursuer} shows an example EZ for the fast Pursuer case along with relevant geometry.
The symbol $\partial$ is used to denote the boundary of a set or region.

\begin{figure}[t]
    \centering
    \tikzsetnextfilename{EZ_fast_Pursuer}
    \includegraphics[width=0.8\textwidth]{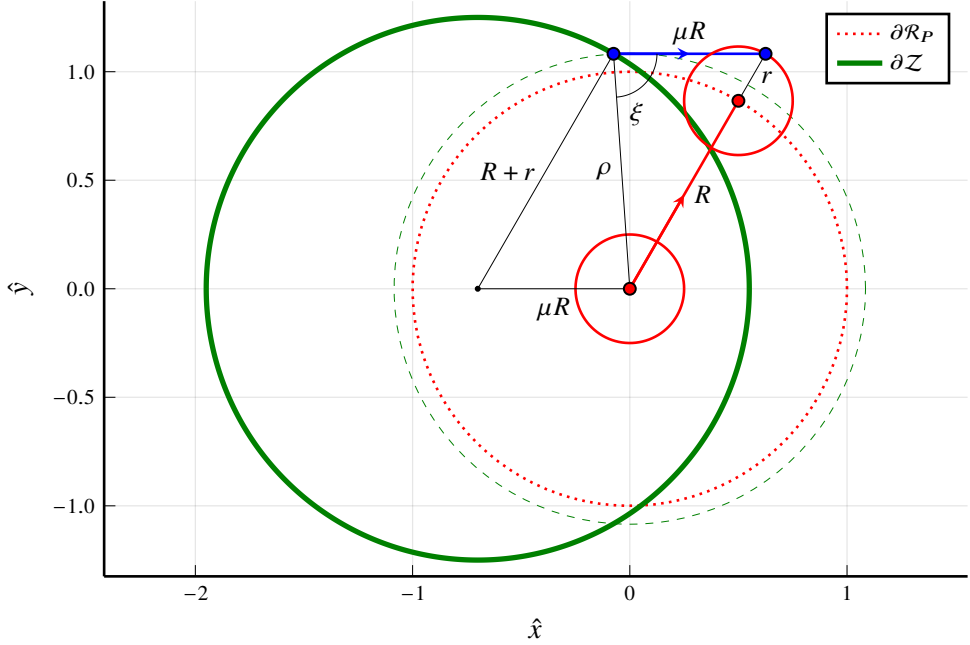}
    \caption{
        Range-limited, simple motion pursuit-evasion Engagement Zone for fixed Agent heading and fast Pursuer (speed ratio $μ = 0.7$) with finite capture radius ($r = 0.25$) and a maximum range of $R = 1$.
    }
    \label{fig:EZ_fast_Pursuer}
\end{figure}

The paper by the authors~\cite{weintraub2022optimal} contains an expression for an EZ that is similar to~\Cref{eq:ρ_pursuit_fast} but is not based on a particular model of the engagement between the Mobile Agent and the Threat.
That expression (with changed notation) is
\begin{equation}
    \label{eq:ρ_pursuit_fast_old}
    \tilde{ρ} = \frac{\cos ξ + 1}{2}\left( \tilde{ρ}_{\max} - \tilde{ρ}_{\min} \right) + ρ_{\min}
\end{equation}
where $\tilde{ρ}_{\max} \equiv \tilde{ρ}(ξ=0)$ and $\tilde{ρ}_{\min} \equiv \tilde{ρ}(ξ=π)$.
The two models for the distance of the EZ boundary from the Threat's initial position may be compared by setting $\tilde{ρ}_{\max} = \max_ξ ρ = \left( 1 + μ \right) R + r$ and $\tilde{ρ}_{\min} = \min_ξ ρ = \left( 1 - μ \right) R + r$.
\Cref{fig:comparison} shows a comparison of~\Cref{eq:ρ_pursuit_fast,eq:ρ_pursuit_fast_old} for a variety of $μ < 1$.
Note that the biggest deviation between the two models occurs for higher $μ$ and for aspect angles near the center of the domain $\lvert ξ \rvert \in \left[ 0, π \right]$.
\begin{figure}[t]
    \centering
    \tikzsetnextfilename{comparison}
    \includegraphics[width=0.8\textwidth, height=4cm]{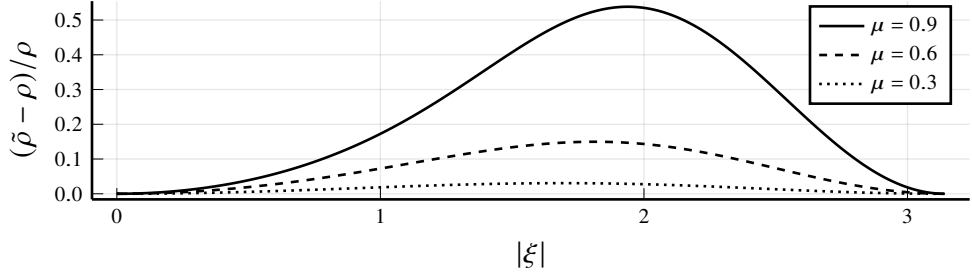}
    \caption{
        Comparison of the Pursuit-Evasion EZ model presented in this paper with a previously developed generic EZ model from~\cite{weintraub2022optimal}.
    }
    \label{fig:comparison}
\end{figure}

This section is focused on the particular interpretation of the EZ that is based on fixing $ψ_A$ and varying $ξ$ to construct the region $\mathcal{Z}$.
Another valid interpretation of the EZ is based on fixing $ξ$ and varying $ψ_A$.
The same geometry depicted in~\Cref{fig:EZ_fast_Pursuer} applies for the example shown.
Based on~\Cref{eq:ρ_pursuit_fast} it is clear that if $ξ$ is held constant $ρ$ is constant as it does not depend directly on $ψ_A$.
Thus, the EZ for a fixed aspect angle is a circle centered at $P_0$, as shown by the dashed green circle in~\Cref{fig:EZ_fast_Pursuer}.
These two interpretations of the EZ boundary yield the same $A_0$ for the points corresponding to the $ξ$ and $ψ_A$ values used.

\subsubsection{Slow Pursuer}

When the Pursuer is relatively slow ($μ > 1$) then there are aspect angles ($ξ$ values) for which neutralization is possible only if the Mobile Agent started on or within the capture disk of the Pursuer.
This is because of the fact that in order for the Pursuer to move and neutralize the Mobile Agent their separation distance must be non-increasing at the time of neutralization.
Thus the optimal Pursuer heading must be $\cos ψ_P^* \leq \cos^{-1} \tfrac{1}{μ}$.
This is akin to the Usable Part of the terminal surface in the verbiage of Isaacs~\cite{isaacs1965differential}.
An analogous result appears in the related scenario covered in~\cite{weintraub2023surveillance} in which this condition is used to delineate the positions from which a slow Observer may establish contact with a fast Target.
The locus of initial Mobile Agent positions corresponding to collision course with $ψ_P^* = \cos^{-1} \tfrac{1}{μ}$ (with the Pursuer travelling $R$ distance or less) forms a portion of the EZ boundary.
From the Law of Sines, the Pursuer's travel distance may be expressed as
\begin{equation}
    \label{eq:PP_f}
    \overline{P_0P_f} = \frac{r \sin (ξ + ψ_P)}{μ \sin ξ - \sin(ξ + ψ_P)}.
\end{equation}
Invoking the Law of Sines once more and substituting in the above with $ψ_P = \cos^{-1} \tfrac{1}{μ}$ gives
\begin{equation}
    \label{eq:ρ_pursuit_slow_touchandgo}
    ρ(ξ; μ, R, r) = \frac{r \sqrt{μ^2 - 1}}{μ \sin \lvert ξ \rvert - \sin\left( \lvert ξ \rvert + \cos^{-1} \tfrac{1}{μ} \right)}.
\end{equation}
It can be verified that the aspect angle which satisfies both~\Cref{eq:ρ_pursuit_fast,eq:ρ_pursuit_slow_touchandgo} is
\begin{equation}
    \label{eq:ξ_crossover}
    ξ_c = \pm \sin\left( \frac{\left( R + r \right) \sqrt{μ^2 - 1}}{μ R \sqrt{μ^2 - 1 + \frac{r^2}{R^2}}} \right).
\end{equation}
Of course, the last possible $ξ$ value in this family of trajectories corresponds to the one where $\overline{P_0P_f} = 0$ which gives $ξ_{\max} = π - \cos^{-1}\tfrac{1}{μ}$.
Finally, the complete expression for $ρ$ in the case of a slow Pursuer is given by
\begin{equation}
    \label{eq:ρ_pursuit_slow}
    ρ(ξ; μ, R, r) = \begin{cases}
        μ R \left[ \cos ξ + \sqrt{\cos^2 ξ - 1 + \frac{\left( R + r \right)^2}{μ^2 R^2}} \right] & \text{ if } \lvert ξ \rvert \in \left[ 0, ξ_c \right] \\
        \frac{r \sqrt{μ^2 - 1}}{μ \sin \lvert ξ \rvert - \sin\left( \lvert ξ \rvert + \cos^{-1} \frac{1}{μ} \right)} & \text { if } \lvert ξ \rvert \in \left( ξ_c, π - \cos^{-1}\frac{1}{μ} \right] \\
        r & \text{ otherwise}
    \end{cases},
    \qquad μ > 1,\ ξ \in \left[ -π, π \right].
\end{equation}
\Cref{fig:EZ_slow_Pursuer} shows an example for the slow Pursuer case along with relevant geometry.
\begin{figure}[t]
    \centering
    \tikzsetnextfilename{EZ_slow_Pursuer}
    \includegraphics[width=0.8\textwidth]{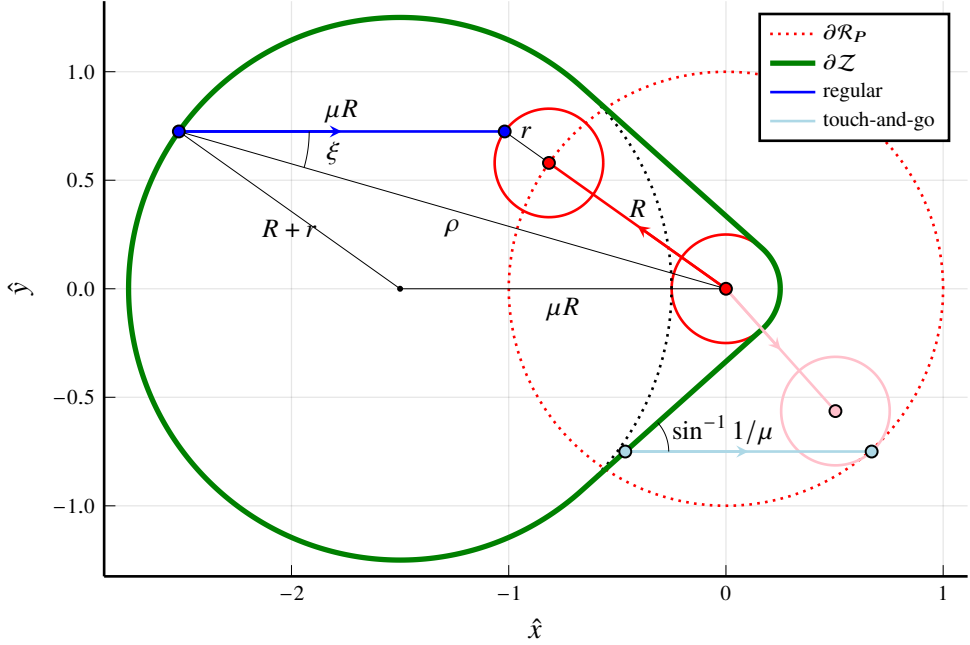}
    \caption{
        Range-limited, simple motion pursuit-evasion Engagement Zone for fixed Agent heading and slow Pursuer (speed ratio $μ = 1.5$) with finite capture radius ($r = 0.25$) and a maximum range of $R = 1$.
        Two example trajectories are shown: (blue/red) the Agent starts on the boundary of the EZ wherein the capture configuration is general, and (light blue/pink) the Agent starts on the EZ boundary wherein the Pursuer-Agent distance rate is zero at capture.
    }
    \label{fig:EZ_slow_Pursuer}
\end{figure}

\section{ Turret-Evasion}
\label{sec:Turret-Evasion}

\subsection{Basic Turret Model}
\label{sec:basic_turret_model}

In this section, the Threat from~\Cref{def:engagement_zone} is specialized to be a Turret, denoted by $T$.
The Turret is a stationary agent with a look angle, $θ$, bounded turn rate, $ω \in \left[ -\bar{\omega}, \bar{\omega} \right]$, and finite range, $R$.
As in the previous section, the Mobile Agent, moves with constant speed and controls its instantaneous heading.
A similar model has been used, e.g., in~\cite{vonmoll2023turret,galyaev2013evading,ivanov1993problem}.
The kinematics of the system in polar form are
\begin{equation}
    \label{eq:f_turret}
\dot{\mathbf{x}} = \begin{bmatrix} \dot{\mathbf{x}}_T \\ \dot{\mathbf{x}}_A \end{bmatrix} = 
\begin{bmatrix} \dot{θ} \\ \dot{x}_A \\ \dot{y}_A \end{bmatrix} =
\begin{bmatrix}
    ω \\
    v_A \cos ψ_A \\
    v_A \sin ψ_A
\end{bmatrix}, \qquad
\dot{\mathbf{x}}(0) =
\begin{bmatrix}
    θ_0 \\
    x_{A_0}\\
    y_{A_0}
\end{bmatrix}
\end{equation}
where, again, the Mobile Agent's control is $u_A \equiv ψ_A$.
Neutralization is said to occur if $\cos θ = \tfrac{x_A}{\overline{AT}}$, $\sin θ = \tfrac{y_A}{\overline{AT}}$, and $\overline{AT} \leq R$.
Define the ratio of the Mobile Agent's speed and Turret's maximum turn rate as $μ = \tfrac{v_A}{\bar{ω}}$.
Note that, without the range constraint on the Turret that neutralization is possible from any initial condition (eventually).
Similar to the range-limited Pursuer considered in the previous section, all the points within a distance of $R$ from the $T$ are considered to be reachable by the Turret, i.e.,
\begin{equation}
    \label{eq:T_reachability}
    \mathcal{R}_T = \left\{ (x,y) \mid \left( x - x_T \right)^2 + \left( y - y_T \right)^2 \leq R \right\}.
\end{equation}

\subsection{Turret EZ}
\label{sec:Turret_EZ}

The EZ constructed for the turret-evasion model is fundamentally different from the one constructed for the Pursuit-Evasion model in the previous section.
That is due to the fact that in the former it was assumed that $P$ was removed from the scenario once its range/fuel had been exhausted, whereas here the Turret is always ``active'' even if it is not moving.
For example, if the Mobile Agent begins very far away and is aimed at the Turret's position it will eventually collide with the Turret's beam thereby becoming neutralized.
Based on~\Cref{def:engagement_zone}, that point, no matter how far away, is considered to be in the EZ.
Therefore, the construction of the turret-evasion EZ is only based upon terminal Agent positions which are exiting the Turret's reachable set, $\mathcal{R}_T$ (i.e., the right half-circle in the $(\hat{x}, \hat{y})$ coordinate system).
For a particular terminal Agent position, $A_f$, it is assumed that $T$ turned in the shortest direction to align with the Agent.
Define the angle traversed by the Turret during an engagement as $γ$, thus from this assumption $γ \in \left[ -π, π \right]$.
Moreover, because only those $A_f$ positions which are exiting $\mathcal{R}_T$ matter the range for $γ$ is even further restricted:
\begin{equation}
    \label{eq:γrange}
    γ \in \begin{cases}
        \left[ -\frac{π}{2}- θ_0,\ 0 \right] \cup \left[0,\ \frac{π}{2} - θ_0 \right] & \text{ if } \cos θ_0 > 0 \\
        \left[ -π,\ \frac{π}{2} - θ_0 \right] \cup \left[ \frac{3π}{2} - θ_0,\ π \right] & \text{ if } \cos θ_0 < 0,\ \sin θ_0 > 0 \\
        \left[ -π,\ -\frac{3π}{2} - θ_0 \right] \cup \left[ -\frac{π}{2} - θ_0,\ π \right] & \text{ if } \cos θ_0 < 0,\ \sin θ_0 < 0
    \end{cases}
\end{equation}
This range for $γ$ gives results in the final Turret look angle to be in the range $θ_f \in \left[ -π/2,\ π/2 \right]$.

\begin{proposition}
    Under the model~\Cref{eq:f_turret} the most limiting scenario (in terms of initial Agent positions, $A_0$) is when $T$ turns at its maximum angular speed and neutralizes $A$ exactly when it is exiting $\mathcal{R}_T$, i.e., $\cos θ_f = \tfrac{x_{A_f}}{R}$ and $\sin θ_f = \tfrac{y_{A_f}}{R}$ where $θ_f \equiv θ_0 + γ$.
\end{proposition}

\begin{figure}[t]
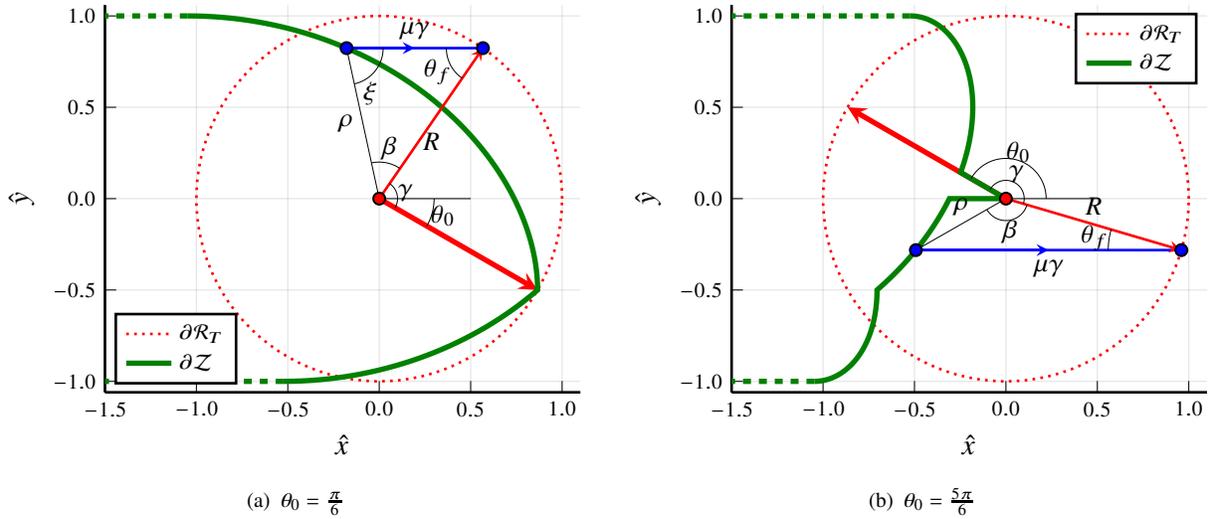

    \centering
    \tikzsetnextfilename{EZ_away_turret}
    \subfigure[$θ_0 = \tfrac{π}{6}$]{\includegraphics[width=0.49\linewidth]{EZ_away_turret}
    \label{fig:EZ_away_turret}}
    \tikzsetnextfilename{EZ_toward_Turret}
    \subfigure[$θ_0 = \tfrac{5π}{6}$]{\includegraphics[width=0.49\linewidth]{EZ_toward_Turret}
    \label{fig:EZ_toward_Turret}}
    \caption{
        Range-limited turret-evasion Engagement Zone for fixed Agent heading and Turret (speed ratio $μ = 0.5$) with maximum range of $R = 1$.
    }
    \label{fig:EZ_turret}
\end{figure}

\section{Application to Path Planning}
\label{sec:Application_to_Path_Planning}

One possible application of these results is to plan paths which stay outside of the EZ.
Perhaps the most basic instantiation of such a scenario is to specify initial and goal positions for a vehicle, place a range-limited Pursuer in between, and specify a desire for the vehicle to reach the goal position in minimum time.
This usage of the EZ was the subject of~\cite{weintraub2022optimal,dillon2023optimal}, although, in the former, considerations for acceptance of some entry into the EZ were included.
Mathematically speaking, we wish to solve the following problem:
\begin{equation}
    \label{eq:path_planning}
    \begin{alignedat}{3}
        &\min_{ψ(t)} \qquad && t_f \qquad&& \\
        &\text{subject to} \qquad && \mathbf{x}_{A}(0) = \begin{bmatrix} x_{A_0} & y_{A_0}\end{bmatrix}^\top, \qquad&& \\
        & && \mathbf{x}_{A}(t_f) = \begin{bmatrix} x_{A_f} & y_{A_f}\end{bmatrix}^\top, \qquad&& \\
        & && \left(\mathbf{x}_{A}(t),\ ψ(t)\right) \notin \mathcal{Z}, \qquad &&\forall t \in \left[ 0,\ t_f \right],
    \end{alignedat}
\end{equation}
where $\mathcal{Z}$ is given by~\Cref{eq:Z_pursuit}.
Note that the instantaneous EZ shape depends on the Agent's position as well as its heading.
\Cref{eq:path_planning} is solved by first discretizing the trajectory and then solving via a nonlinear program.
Specifically, even collocation is used (wherein the path constraint of staying outside the EZ is imposed at the discrete nodes) and the Ipopt solver~\cite{wachter2006implementation} is used within the JuMP package~\cite{lubin2023jump} for the Julia programming language.

In order to illustrate some of the potential benefits of the EZ-based path planning approach several circumnavigation-based nominal trajectories are described in the following.
Each circumnavigation-based trajectory has an associated circle, centered on $P$ with a particular radius, which the Agent navigates around in minimum time.
In particular, the Agent heads from its initial position towards the tangent on the associated circle, remains on the circle for some portion of its trajectory, and then departs the circle tangentially to reach the specified goal location.
Therefore, the associated travel time is given by
\begin{equation}
    \label{eq:t_circumnav}
    t_f^\bigcirc = \frac{1}{μ} \left( \sqrt{\overline{PA_0} - \hat{R}^2} + \sqrt{\overline{PA_f} - \hat{R}^2} + \hat{R} \lvert θ_2 - θ_1 \rvert\right),
\end{equation}
where $\hat{R}$ is the radius of the circle and $θ_1$, $θ_2$ are given by
\begin{equation}
    \label{eq:circumnav_θs}
    θ_1 = \atan\left( y_{A_0},\ x_{A_0} \right) - \cos^{-1}\left( \frac{\hat{R}}{\overline{PA_0}} \right), \qquad
    θ_2 = \atan\left( y_{A_f},\ x_{A_f} \right) + \cos^{-1}\left( \frac{\hat{R}}{\overline{PA_f}} \right),
\end{equation}
where $\atan$ is the two-argument inverse tangent function.
Three different circumnavigation radii are considered and summarized in~\Cref{tab:circumnavigation_trajectories}.
The last column in the table gives the improvement (negative) or loss (positive) of the EZ-based path plan w.r.t.\ the associated circumnavigation path plan in terms of time, i.e.,
\begin{equation}
    \label{eq:comparison_percentage}
    \text{EZ \% difference} = \left(\frac{t_f^\textrm{EZ} - t_f^\bigcirc}{t_f^\bigcirc}\right) \cdot 100.
\end{equation}
The trajectories are shown in~\Cref{fig:path-plan}.
\begin{table}[t]
    \centering
    \caption{Circumnavigation Trajectories}
    \label{tab:circumnavigation_trajectories}
    \begin{tabular}{cccc}
        \toprule
        Label & $\hat{R}$ & Time & EZ \% difference \\
        \midrule
        Reach & $R$ & 7.46 & -1.61\% \\
        Worst & $(1 + μ) R + r$ & 8.44 & -13.0\% \\
        Apol & $(1 - μ) R + r$ & 7.04 & 4.32\% \\
        \bottomrule
    \end{tabular}
\end{table}
\begin{figure}[t]
    \centering
    \tikzsetnextfilename{path-plan}
    \includegraphics[width=0.8\textwidth]{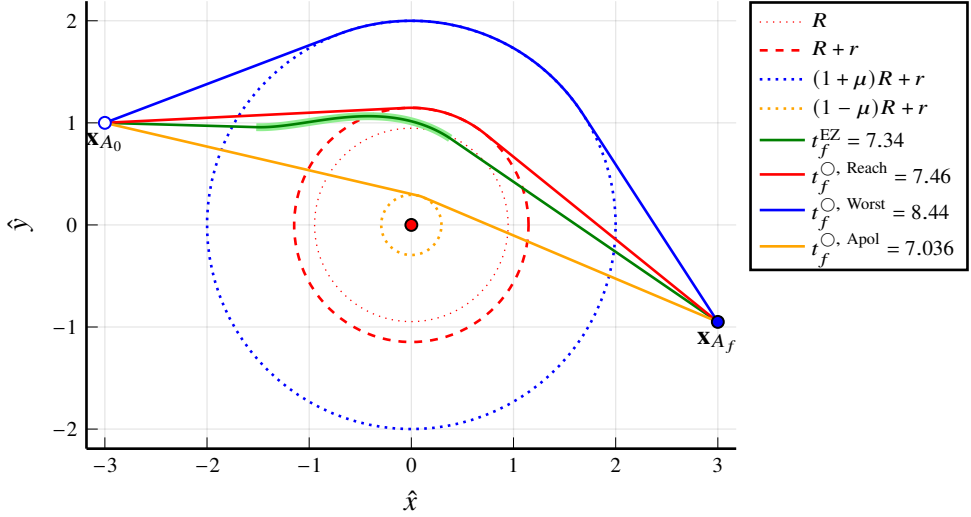}
    \caption{
        Example path planning problem with $μ = 0.9$, $r = 0.2$, and $R = \tfrac{2 - r}{μ + 1} \approx 0.95$, roughly corresponding to the example in~\cite[Fig. 4]{weintraub2022optimal}.
        The highlighted portion of the green EZ trajectory indicates activation of the constraint, i.e., $\left( \mathbf{x}_A(t),\ ψ(t) \right) \in \partial \mathcal{Z}$.
    }
    \label{fig:path-plan}
\end{figure}

The circumnavigation trajectory labeled \textit{Reach} is based on navigating around the circle corresponding to the reachable region of $P$ accounting for its capture radius, i.e. the capturability region, hence $\hat{R} = R + r$.
If the Agent follows this trajectory it need not react to Pursuer if the latter were to begin moving since capture is not possible for any of the points along the trajectory (depending on whether the boundary of the capturability region is considered capture or not).
The EZ-based path plan results in a time savings of 1.61\% w.r.t.\ this path plan due to the fact that the former passes through the capturability region.
For those points inside said region, if $P$ were to begin moving $A$ would be guaranteed to exit the region before capture could occur simply by maintaining its heading (by way of these points being outside the EZ).
Thus $P$ has little to no incentive to begin moving in the first place.

Next, the circumnavigation trajectory labeled \textit{Worst} is based on the distance $A$ would have to be from $P$ such that if $A$ were heading directly to $P$ when the latter began moving that $A$ would just barely avoid capture, hence $\hat{R} = (1 + μ) R + r$.
This trajectory is overly conservative because at no point along the trajectory is $A$ ever pointed towards $P$.
Nonetheless, this trajectory may be employed in practice due to the relative ease in estimating this worst-case distance as opposed to computing a more realistic EZ shape (e.g., in higher fidelity scenarios where the EZ can only be computed via a computationally expensive numerical procedure).
Additionally, safety is guaranteed \textit{irrespective} of the Agent's heading.
This is in contrast with the \textit{Reach} trajectory since the Agent's heading \textit{must} point tangentially or outside the capturability region when it is on the boundary.
The EZ-based path plan results in a time savings of 13\% w.r.t.\ this path plan.

Finally, the circumnavigation trajectory labeled \textit{Apol} is based on the distance $A$ would have to be from $P$ such that $A$ could just barely escape from $P$, hence $\hat{R} = (1-μ) R + r$.
This condition is described in detail in~\cite{weintraub2023range-limited}.
Navigating along this trajectory is inherently riskier as it may require $A$ to perform an extreme maneuver in order to avoid being captured by $P$ (unlike any of the other trajectories considered).
In fact, when $A$'s position is such that $\overline{AP} = (1 - μ) R + r$ there is only one heading which can guarantee escape: heading directly away from $P$.
Depending on where $A$ is located when $P$ begins moving, performing this evasive maneuver could lengthen $A$'s trajectory considerably.
This is a high price to pay considering that the EZ-based path plan is only 4.32\% longer.
The degree of lengthening that the Pursuer can achieve on the Evader's path is out of the scope of this study but will be considered in future work.
One instance of that problem has been considered, for example, in~\cite{zhang2022open}.

\section{Conclusion}
\label{sec:conclusion}

This paper has established a more formal definition for an Engagement Zone and derived some basic EZs associated with fundamental engagement models associated with pursuit-evasion and turret-evasion.
The basic EZs presented in this paper capture the most salient aspects of the Pursuer-Agent and Turret-Agent engagements: namely the geometry of the aspect angle and the relative differences in capability (i.e., maximum speeds, range, etc.).
One of the main advantages of utilizing EZs for path planning is that they encode an overall desire for Agent to go \textit{somewhere} without requiring an aggressive maneuver or active evasion should the Pursuer or Turret begin its pursuit.
It was shown that there is some advantage in terms of time savings in EZ-based navigation around a single range-limited Pursuer as compared with circumnavigating the capturability region.
Future work will focus on investigating worst-case path lengthening in the event that the Threat decides to engage with the Agent while it is \textit{en route}.

\bibliography{alexbib}

\end{document}